\documentclass[12pt]{article}
\usepackage{amsmath,amsfonts,amssymb,amsthm,amscd}
\title{On $pp$-elimination and stability in a continuous setting}
\date{\today}
\author{Nicolas Chavarria Gomez\\ University of Notre Dame \and Anand Pillay\thanks{Supported by NSF grants DMS-1665035, DMS-1760212, DMS-2054271, and a SDV at the Fields Institute, Toronto.}\\University of Notre Dame}

\newtheorem{Theorem}{Theorem}[section]
\newtheorem{Proposition}[Theorem]{Proposition}
\newtheorem{Definition}[Theorem]{Definition}
\newtheorem{Remark}[Theorem]{Remark}
\newtheorem{Lemma}[Theorem]{Lemma}
\newtheorem{Corollary}[Theorem]{Corollary}
\newtheorem{Fact}[Theorem]{Fact}
\newtheorem{Example}[Theorem]{Example}

\newtheorem{Problem}[Theorem]{Problem}

\newcommand{\Z}{\mathbb Z}

\begin{document}
\maketitle

\begin{abstract}    We generalize ``$pp$-elimination" for modules, or more generally, abelian structures, to a continuous logic setting where the abelian structure is equipped with a homomorphism to a compact (Hausdorff) 
group.  We conclude that the continuous logic theory of such a structure is stable. 

\end{abstract}

\section{Introduction}
The model theory of modules, or more generally abelian structures, comprises an important chapter in model theory and mathematical logic, feeding into geometric stability theory  (see Chapter 4 of \cite{Pillay-GST}) as well as representation theory \cite{Prest}.

Among the key facts in the context of left $R$-modules, for $R$ a unitary ring, and the corresponding language $L_{R}$, is $pp$ elimination: for a given module $M$, any formula $\phi(\bar x)$ is equivalent in $M$ to a (finite) Boolean combination of $pp$ (positive primitive) formulas.

%A closely related fact, which follows from the proof of $pp$-elimination is that the complete theory of an $R$-module $M$ is determined by its``Baur-Monk" invariants, i.e. the index of $\phi(x)%\wedge \psi(x)(M)$ in $\phi(x)(M)$  (a fixed positive finite integer, or $\infty$) for $\phi(x)$, $\psi(x)$, positive primitive formulas {\em in one free variable}. 

In unpublished work, Fisher introduced the more general notion of an {\em abelian structure} as a many sorted structure, where each sort has an abelian group structure, and the distinguished relations are subgroups of Cartesian powers of sorts, and where similar theorems hold. In particular one obtains stability of the first order theories of abelian structures,  in a very strong form which has come to be known as $1$-basedness.  Moreover any one-based group is essentially an abelian structure. 

It is natural to ask what, if anything, is the analogue of an abelian structure in a continuous logic framework where formulas are real-valued rather than Boolean valued; in particular, what is the continuous logic analogue of  a one-based group.  Moreover, are there analogues of $pp$-elimination?
This is what we try to answer in the current paper.   
%The analogues of the Baur-Monk invariants will be treated lin subsequent work by the first author. 

In addition to the motivation described above, 
some other inspiration came from Hrushovski's recent work  on the theory of finite fields with an additive character, as a theory in continuous logic \cite{Hrushovski-additive}.  Hrushovski proved (among many other things) simplicity of the theory, so it was natural to imagine that with only the additive structure on the field we would have stability. 

Our set-up will, for simplicity of presentation,  be a one-sorted abelian structure $A$, that is an abelian group $(A,+, -, 0)$ equipped with a collection $\mathcal S$ of subgroups of various Cartesian powers of $A$ (including $=$, the diagonal in $A^{2}$), {\em together with} a homomorphism $f$ from $A$ to a compact (Hausdorff) group $\mathbb T$ which we may and will assume to be also commutative, and which we will write additively.  Our results will generalize without  much trouble to that of a many sorted abelian structure equipped with a compatible collection of homomorphisms from the sorts to compact Hausdorff groups. 
We write this ``structure" as $M = (A,+,-,0,P,f,{\mathbb T})_{P\in {\mathcal S}}$.  The precise formalism will be described in detail in the next section, where we give a continuous logic appropriate for the study of a first order structure equipped with a map to a compact space, and will be closer to \cite{H-I} than to  \cite{BY-B-H-U}, although ultimately equivalent. 

 For now we will just give an informal description of the results and notions.

The first order logic (FOL) part of the structure is  obtained by forgetting about $\mathbb T$ and $f$, and is just the abelian structure  $M^{-} = (A,+,-,0,P)_{P\in {\mathcal S}}$, in the appropriate language $L^{-}$.  The FOL-atomic formulas  are of the form $P({\bar t}({\bar x}))$, where $P\in {\mathcal S}$ and ${\bar t}$ is a sequence of terms.  The {\em positive primitive} ($pp$) formulas are by definition formulas of the form $\exists{\bar y}\phi({\bar x}, {\bar y})$ where 
$\phi({\bar x}, {\bar y})$ is a (finite) conjunction of FOL-atomic formulas (and we witness all free variables which occur).  So we repeat that in so far as this FOL-structure is concerned, the classical results say that every formula is equivalent, in $M^{-}$ or modulo the first order theory of $M^{-}$, to a Boolean combination of positive primitive formulas  (see \cite{Gute} for example). 
By a $pp^{*}$-formula for $M$ we will mean something of the form $\exists {\bar y}\phi({\bar x}, {\bar y})$ where now $\phi({\bar x}, {\bar y})$ is a finite conjunction of FOL-atomic formulas and ``expressions" of the form $f(x_{i}) = c$, $f(y_{j}) = c$ for $c\in {\mathbb T}$ and $x_{i}, y_{j}$ variables (ranging over $A$) from ${\bar x}$, ${\bar y}$ respectively. 

Note that a $pp$ formula $\phi({\bar x})$ defines a subgroup of the relevant Cartesian power $A^{n}$ of $A$, which we call a $pp$-subgroup.   
If $\phi({\bar x})$ is a $pp^{*}$-formula and all the parameters $c$ from $\mathbb T$ which appear in the formula are $0$, then again $\phi({\bar x})$ defines a subgroup of the relevant $A^{n}$, which we call a $pp^{*}$-subgroup of $A^{n}$. In general a $pp^{*}$-formula will define a {\em coset} of a $pp^{*}$-subgroup.   

We will be defining syntax and semantics, in particular the language $L$ of the structure $M$, the notions of continuous logic (CL) $L$-formulas, saturation, etc. All $FOL$ formulas in the language $L^{-}$ of $M^{-}$ will be continuous logic $L$- formulas, and $pp^{*}$-formulas will  also be  such formulas.   

We will say that the ``continuous theory" $Th_{CL}(M)$ is {\em stable} if whenever $({\bar a}_{i}, {\bar b}_{i})$ is an indiscernible sequence in a saturated model and $i<j$ then $({\bar a}_{i}, {\bar b}_{j})$ has the same type as $({\bar a}_{j}, {\bar b}_{i})$. 

Our results, expressed somewhat informally, are: 
\begin{Theorem}  (informal version.) Assume $M$ to be saturated (in the appropriate sense).  Then two finite tuples with the same $pp^{*}$-type have the same type. We conclude that $Th_{CL}(M)$ is stable, as well as every formula being equivalent (in a suitable approximate sense) to a suitable combination of $pp^{*}$-formulas and negated $pp$-formulas. 

\end{Theorem}

The precise statements, in  light of the ``theory" developed in the next section, appear in  Theorem 3.3,  Corollary 3.8, and Remark 3.7.

\vspace{5mm}
\noindent
 Thanks to Ward Henson, Ehud Hrushovski and Tommy Kucera for communications about continuous logic and the material in this paper. 
Kucera pointed out the possible connection to his own work on the model theory of topological modules in the Flum-Ziegler logic $L_{t}$ \cite{Kucera}.

\section{Continuous logic}
Here we will introduce a version of continuous logic suitable for our purposes and for possible future work on related problems.  The expression ``continuous logic" has come to be identified with the  formalism introduced in \cite{BY-B-H-U} and \cite{BY-U}.  Earlier versions include the theory developed by Chang and Keisler \cite{C-K}, and  Henson's ``approximate logic" or ``positive bounded logic"  \cite{H-I}.  Ben-Yaacov's compact abstract theories (CATS) \cite{BY-CATS}, ``positive logic", or ``positive Robinson theories", are somewhat more general (as explained in \cite{BY-U}).  Robinson theories \cite{Hrushovski-Robinson} as well as the extension in \cite{Pillay-ec} are special cases of CATS,  all of which also appear in some form (not always explicitly) in \cite{Shelah-Lazy-Guide}. 

%There is also the  ``positive logic  and compact abstract theories" of Ben-Yaacov (\cite{BY-CATS}), which we will not say much about, although it is relevant.  

%In Chang and Keisler's monograph, \cite{C-K}, formulas take values in compact Hausdorff spaces $C$.   The set up (both syntax and  semantics) is rather hard to parse. I am unaware of any work by others which makes use of this set-up.  However, the monograph built on  earlier work by Chang, which as far as we understand was influenced by Lukasiewicz's work on  infinite valued propositional logic. 

%Henson's positive bounded logic is a logic for expansions of normed vector spaces. There are two sorts, the vector space sort $V$ and the reals $\R$.  Formulas are either true or false. As far as we understand one of the motivations was to find the right logic for formulating a Los theorem appropriate for ``metric ultraproducts". 

In the Ben-Yaacov, Berenstein, Henson, and Usvyatsov version \cite{BY-B-H-U}, \cite{BY-U},  relations and formulas have values in the real unit interval $[0,1]$, and structures are also equipped with a metric $d$ (with values in $[0,1]$) which  replaces the equality relation and coincides with it when the metric is discrete.  
They also talk about ``conditions", which are typically of the form that a given formula has a given value, and these conditions {\em are} true or false.  Whereas in \cite{H-I} the formulas themselves are true or false. 

Roughly speaking, in continuous logic,  compact spaces $C$ such as $[0,1]$ have a privileged role, as repositories of truth values (generalizing True/False) or ``distances", or as sorts which are fixed as models or structures vary.

We are here interested in  ``continuous logic"  structures $M$ of the form $(M^{-},f,C)$ where $M^{-}$ is a classical first order structure (with equality)  and $f$ is a map from the universe of $M^{-}$ to a compact Hausdorff space $C$.  This extends  to the context where $M^{-}$ is many sorted and we have a collection of maps from Cartesian powers of various sorts of $M^{-}$ to various compact spaces.    

%We could use the tools of \cite{BY-B-H-U}, in at least 2 ways.
%\newline
%(i)  Replace the map $f$ from $M^{-}$ to $C$ by the collection of all maps from $M^{-}$ to $[0,1]$ which are the composition of $f$ with a continuous function from $C$ to $[0,1]$. And view $M^{-}$ equipped with all these maps as ``$[0,1]$-valued relations", and with the first order relations on $M^{-}$ treated as $\{0,1\}$-valued relations, and with the discrete metric on $M^{-}$, as a continuous logic structure in the sense of \cite{BY-B-H-U}.
%\newline
%(ii) (as pointed out by Henson)  
%\newline
%In the special case where $C$ is a metric space, then view $M$ as a $2$-sorted structure where the metric on $M^{-}$ is discrete, where the metric on $C$ is the given one, $f$ is a distinguished function between the sorts, and all continuous functions from $C$ to $[0,1]$ are real-valued relations on the $C$-sort. 
%\newline
%For general $C$, view it as an inverse limit of metrizable compact spaces, and treat all of these as new sorts etc....

\vspace{2mm}
\noindent
Because our structure is mixed  (a classical first order part and a continuous part), and also because our immediate aim is to generalize the usual $pp$-elimination theorem,  it will be conceptually and technically convenient to consider $(M^{-},f,C)$ as a  two-sorted structure, where  formulas are true or false.   $C$ could be equipped with {\em all} subsets of its Cartesian powers although we will need less.  $C$ will not be allowed to vary, our relevant ``continuous logic" formulas will be a subset of all the first order formulas, and the semantics will be induced by the usual semantics. 

The set-up will be close to that of Henson and Iovino \cite{H-I} in various ways, although our underlying structure $M^{-}$ is not a normed vector space.  

The proofs of the basic model theory results in this section are completely routine, but we may give some hints, sketches,  or references to \cite{H-I} for similar things.  The acronym FOL stands for (classical) first order logic, and CL will stand for continuous (first order) logic, which is denoted by CFOL in some other places. 

At the end of this section we will make a comparison or translation between our formalism and that of \cite{BY-B-H-U}, aimed at those who are familiar with the latter. 

We will start with the ``structure" $M$, produce a language and then define what other CL structures for the logic are, and then define CL-formulas and the semantics. 

So let us fix  the structure $M = (M^{-}, f, C)$.  

Let $L^{-}$ be the language or vocabulary of the one sorted structure $M^{-}$ (including equality) as a FOL structure. 

$L$ will be a $2$-sorted language, a sort $P$ for $M^{-}$ and a sort $Q$ for $C$.  $L^{-}$ will be a sublanguage of $L$, with the proviso that all the symbols are restricted to the $P$-sort. 
We will have a function symbol (without loss of generality $f$) for the function $f:M^{-} \to C$.  We also have predicate symbols $P_{D}$ for every closed subset $D$ of $C^{n}$ (as $n$ varies), in particular a predicate for equality on $C$.   Note that in particular we have essentially constant symbols for each element $c$ of $C$ (represented as the singleton $\{c\}$).  

In an arbitrary  FOL $L$-structure, the $P$ sort will be an $L^{-}$-structure, and the $Q$ sort will be equipped with the interpretations of the symbols $P_{D}\subset Q^{n}$. But:

\begin{Definition} By a CL $L$-structure based on the compact space $C$ we mean an FOL $L$-structure such that the interpretation of the $Q$-sort is precisely  $C$ with the tautological interpretations of  predicate symbols $P_{D}$.  In particular $M$ itself is a CL $L$-structure. 
\end{Definition} 

In our context we  may suppress both $L$ and $C$ and just talk about CL-structures.  

Let us note in passing that the cardinality of $L$ will be greater than the cardinality of $C$.  We could economize a bit by working with a smaller collection of closed subsets of $C^{n}$ (satisfying suitable denseness properties), but in the interests of a simple presentation we will ignore this for now. 

Note that all the $L$-terms have domain in some $P^{n}$, and those that are $Q$-valued consist of the composition of an $L^{-}$-term with $f$.
We now define the CL $L$-formulas, a subset of the collection of FOL $L$-formulas.  The only variables appearing in the CL formulas, whether free or quantified,  will be of sort $P$. 

\begin{Definition} 
(i)  Any FOL $L^{-}$-formula is a CL formula.
\newline
(ii) Suppose that $t_{1},..,t_{n}$ are $Q$-valued terms  and $D$ is a closed subset of  $C^{n}$, then $P_{D}(t_{1},..,t_{n})$, which we may also write as
$(t_{1},..,t_{n})\in D$,  is a CL-formula.  If $D$ is a singleton $\{(c_{1},..,c_{n})\}\in C^{n}$, we may write this formula as ${\bar t} = {\bar c}$. 
\newline
(iii) If $\phi$, $\psi$ are CL-formulas, so are  $(\phi \wedge \psi)$ and  $(\phi \vee \psi)$.
\newline
(iv) If $\phi$ is a CL-formula and $x$ is a variable of sort $P$ then $\exists x \phi$ and $\forall x \phi$ are CL-formulas.
\newline
(v) Nothing else is a CL-formula. 
\newline
(vi)  a CL sentence is a CL formula without free variables. 
\end{Definition}

So there are no negations in CL-formulas other than inside $L^{-}$-subformulas. 

\begin{Remark}
As a CL-structure is also a (FOL) $L$-structure and a CL-formula is an FOL $L$-formula, the usual satisfaction relation $N\models \phi({\bar a})$ where $N$ is an $L$-structure, $\phi({\bar x})$ an $L$-formula, and ${\bar a}$ a suitable tuple from $N$, specializes to the case where $N$ is a CL-structure and $\phi({\bar x})$ a CL-formula, and ${\bar a}$ a tuple (from the $P$-sort of $N$).  So there is no need for a separate definition of truth for CL-formulas and CL-structures. 
\end{Remark}

In usual model theory, especially stability, we are often concerned just with behaviour and definability in a saturated model.  For the CL-version, the above definitions suffice for working in ``saturated models", but to work with ``arbitrary models" (or even to define what we mean by saturation) we need the notion of approximate truth and satisfaction, as in \cite{H-I}.  This is really about the quantifiers. In continuous logic in the sense of \cite{BY-B-H-U} where formulas have values in $[0,1]$, approximate truth is directly built into  the set-up via having all continuous functions from $[0,1]^{n}$ to $[0,1]$ as connectives and  {\em inf} and {\em sup} as the quantifiers.  This will be discussed later.

The following is an adaptation of some of the definitions in Chapters 5  and 6 of \cite{H-I} to our context.  It is a bit long-winded, but straightforward.  It is very similar to the situation with hyperimaginaries in FOL where we only really see them in saturated models, and one can ask what exactly is discerned about them in arbitrary models. 

\begin{Definition} (i) Let $D$ be a closed subset of $C^{n}$.  By an approximation to $D$, we mean a closed neighbourhood $D'$ of $D$, namely a closed subset of $C^{n}$ such that for some open subset $U$ of $C^{n}$, we have that $D\subseteq U \subseteq D'$. 
\newline
(ii) Let $\phi$ be a CL-formula. By an approximation to $\phi$ we mean a CL-formula obtained from $\phi$ by replacing every appearance of any closed set $D$ in $\phi$ by an approximation to $D$. The formal inductive definition is similar to the definition at the bottom of p. 24 and top of p. 25 in \cite{H-I}.
\newline
(iii) For $\Gamma$ a set of CL-formulas, closed under finite conjunctions, let $\Gamma^{+}$ be the set of approximations of formulas in $\Gamma$ (so with the same free variables). 
\newline
(iv) Let $\Gamma({\bar x})$ be a collection of CL-formulas with free variables $\bar x$, let $N$ be a CL-structure, and ${\bar a}$ a tuple from $N^{-}$ of the relevant size. We say that  ${\bar a}$ approximately satisfies (or realizes) 
$\Gamma({\bar x})$ in $N$, $N\models_{approx} \Gamma({\bar a})$,  if  $N\models \phi({\bar a})$ for all $\phi({\bar x})\in \Gamma^{+}$.
\newline
(v)  Two CL-structures $N_{1}$, $N_{2}$ are said to be ``approximately CL-elementarily equivalent" if they approximately satisfy the same CL sentences.
\newline
(vi) Finally we say (for CL-structures $N_{1}$, and $N_{2}$) that $N_{1}$ is an approximate elementary substructure of $N_{2}$, written $N_{1}\prec_{approx} N_{2}$, if for every CL-formula $\phi({\bar x})$ and ${\bar a}$ a tuple from $N_{1}$, $N_{1}\models_{approx} \phi({\bar a})$ iff $N_{2}\models _{approx} \phi({\bar a})$. 
\end{Definition} 

Note that if $\phi$ is a FOL $L^{-}$ formula then (by definition or convention) the only approximations of $\phi$ are itself.  Also if $D$ is a clopen subset of $C^{n}$, then $D$ is an approximation of itself.

From this point on $M$ will denote an arbitrary CL $L$-structure based on the compact space $C$.

\begin{Remark} Let $M$ be a CL-structure, $\bar a$ a tuple from $M$ and $\phi({\bar x})$  a CL-formula.
\newline
(i) If $M\models \phi({\bar a})$ then $M\models_{approx}\phi({\bar a})$.
\newline
(ii)  If $\phi$ is quantifier-free then $M\models \phi({\bar a})$ if and only if $M\models_{approx}\phi({\bar a})$.
\end{Remark}

%Let us make the brief connection with the $inf$ and $sup$ quantifiers from the  \cite{BY-B-H-U} formalism.  Suppose that in our current formalism $M = (M^{-},f,C)$ is a CL-structure based on $C$ where $C = [0,1]$. Let $t(x)$ be a $C$-valued term. Then  $M\models_{approx} \exists x (t(x) = 0)$ iff $inf_{a\in M^{-}}t(a) =  0$, and  $M\models_{approx} \forall x (t(x) = 0)$ iff $sup_{a\in M^{-}}t(a) = 0$.

Given a CL-structure $N$ we can introduce constant symbols for some set $A\subseteq N^{-}$ and talk about CL-formulas over $A$.

\begin{Proposition}
Let $N$ be a CL-structure, and $\Gamma({\bar x})$ be a collection of CL-formulas over $N^{-}$ which is finitely approximately (or approximately finitely) satisfiable in $N$, namely for every finite subset $\Gamma'$ of $\Gamma$, we have $N\models_{approx} \exists {\bar x}( \bigwedge\Gamma'({\bar x}))$. Then there is an approximate elementary extension $N'$ of $N$ and some tuple ${\bar a}$ from $N'$ which realizes $\Gamma({\bar x})$ in $N'$.
\end{Proposition}
\begin{proof}  Let $\Gamma^{+}$ be the collection of approximations to the collection of finite subsets of $\Gamma$.
Consider $N$ as a FOL $L$-structure. Consider $\Gamma^{+}$ as a set of $FOL$ $L$-formulas over $N$. By FOL 
compactness we find a sufficiently saturated FOL elementary extension $N^{*}$ of $N$ and ${\bar a}$ in $N^{*}$ which 
realizes $\Gamma^{+}$.  Let $C^{*}$ be the interpretation of the $Q$ sort in $N^{*}$ and $f^{*}$ the 
interpretation in $N^{*}$ of $f$ (from the $P$ sort to the $Q$ sort).  As $C^{*}$ is equipped with all the $L$-structure 
coming from $C$ we have the standard part map $\pi$ from $C^{*}$ onto $C$. Let $f'$ be the composition of $f^{*}$ 
with $\pi$ from $(N^{*})^{-}$ to $C$. We let $N' = ((N^{*})^{-},f,C)$ as a CL-structure. We have to show that
\newline
(i)  $N'$ is an approximate elementary extension of $N$ and
\newline
(ii) ${\bar a}$  
(which is in $N'^{-}$) realizes  $\Gamma({\bar x})$ in $N'$. 

\vspace{2mm}
\noindent
(i) is left to the reader and only uses that $N^{*}$ is an FOL elementary extension of $N$  (not the saturation of $N^{*}$).
\newline
We now consider (ii) which {\em does} use saturation of $N^{*}$ and we will do just a special case, namely that of a single formula of the form  $\phi({\bar x})$:  $\exists {\bar y} (f({\bar x}, {\bar y})\in D)$ where ${\bar x}$ is an $n$-tuple, ${\bar y}$ a $k$-tuple, and $D$ a (predicate symbol for) a closed subset of $C^{n+k}$.
The approximations to $\phi({\bar x})$ have the form $\phi'({\bar x})$: $\exists {\bar y}(f({\bar x}, {\bar y}) \in D')$ where $D'$ is an approximation to $D$. Note that a finite intersection of approximations to $D$ is also an approximation to $D$.  By the assumption that ${\bar a}$ realizes $\Gamma^{+}$ in $N^{*}$ and saturation of $N^{*}$ there is ${\bar b}$ in $(N^{*})^{-}$ such that $f^{*}({\bar a}, {\bar b}) \in (D')^{*}$ for all approximations $D'$ to $D$.
But $D' = \pi((D')^{*})$ for each such approximation $D'$ of $D$ (as $D'$ is compact). Hence $f'({\bar a}, {\bar b})\in D'$ for each approximation $D'$ of $D$, so $f'({\bar a}, {\bar b})\in D$, whereby $N'\models \phi({\bar a})$. 

 The case of general $\phi\in \Gamma$ is proved by induction, but this existential case is the key step, and the ``positive" nature of the logic is precisely set up to allow the above argument to go through.
\end{proof}

\begin{Definition} The CL-structure $M$ is said to be  $\kappa$-saturated, if  whenever $\Gamma({\bar x})$ is a collection of CL-formulas over a subset $A$ of $M^{-}$ of cardinality $<\kappa$ and 
$\Gamma$ is finitely approximately satisfiable in $M$, then $\Gamma$ is satisfiable (or realized) in $M$.
\end{Definition}

\begin{Proposition} (i) Every CL-structure has a $\kappa$-saturated approximate elementary extension.
\newline
(ii) Suppose  $N^{*}$ is a $\kappa$-saturated FOL $L$-structure, and $N'$ is the CL-structure obtained from $N^{*}$ by applying the standard part map to the $Q$ sort (as in the proof of Proposition 2.6). Then $N'$ is $\kappa$-saturated as a CL-structure.
\newline
(iii) If $M$ is an $\omega$-saturated CL-structure and $\phi({\bar x})$ is a CL-formula and ${\bar a}$ a tuple from $M^{-}$ then $M\models_{approx} \phi({\bar a})$ iff $M\models \phi(\bar a)$. 
In particular for $\sigma$ an $L$-sentence $M\models_{approx} \sigma$ iff $M\models\sigma$

\end{Proposition} 
\begin{proof}  (i) and (ii) follow from Proposition 2.6 and its proof.
The proof of (iii) is routine and by induction on the complexity of $\phi({\bar x})$. 

\end{proof}

We will be mainly interested in saturated CL-structures, but in the background are ``complete" CL-theories which are discussed now.  By a CL-theory we will mean a set of CL-sentences with a CL-model. The CL-theory $Th_{CL}(M)$ of a CL-structure is such a CL-theory, but need not be ``complete". 

\begin{Definition} (i) Let $M$ be a CL-structure. By the approximate CL-theory of $M$, $Th_{CL, approx}(M)$ we mean the set of CL-sentences $\sigma$ such that
$M\models_{approx} \sigma$.
\newline
(ii) Likewise if $T$ is a CL-theory, by an approximate model of $T$ we we mean a CL-structure $N$ such that $N\models_{approx} \sigma$ for all $\sigma\in T$. 
\end{Definition} 

Of course two CL-structures are approximately elementarily equivalent if they have the same approximate CL-theory.

\begin{Lemma}  Let $M$ be a CL-structure and $T = Th_{CL,approx}(M)$. Then $T$ is a maximal CL-theory, and every maximal CL-theory is of this form.
\end{Lemma}
\begin{proof} 
For the first part, suppose $\sigma$ is a $CL$-sentence such that  it is not the case that $M\models_{approx}\sigma$. We will show that $Th_{CL, approx}(M)\cup\{\sigma\}$ is inconsistent in the sense of having no $CL$ model (in fact no model).
We will just consider the special case where $\sigma$ has the form $\exists x (\rho(x)\wedge f(x)\in D)$ where $\rho(x)$ is a $FOL$ $L^{-}$-formula.  By assumption for some approximation $D'$ to $D$, $\exists x(\rho(x)\wedge f(x)\in D')$ is 
not true in $M$. Now for some open set $U$ in $C$ we have $D\subseteq U \subseteq D'$, whereby $M\models \forall 
x(\rho(x)\to x\in U^{c})$ where $U^{c}$ is the complement of $U$.  This latter CL-sentence is clearly inconsistent with $\sigma$.

Now for the second part: Suppose that $T$ is maximal, and let $M$ be a CL-structure which is a model of $T$.  Let $M\models_{approx}\sigma$. Then by Proposition 2.8, there is a CL-model $N$ of $T$ such that $N\models \sigma$, so by maximality $\sigma\in T$. Hence $T = Th_{CL, approx}(M)$.

\end{proof} 

\begin{Remark} More generally, if $M$ is a CL-structure and ${\bar a}$ is a tuple from $M^{-}$ and $\phi({\bar x})$ is a CL-formula such that it is not the case that $M\models_{approx}\phi({\bar a})$, then there is a CL-formula $\psi({\bar x})$ such that $M\models \psi({\bar a})$, and $(\phi({\bar x})\wedge \psi({\bar x}))$ is inconsistent. 
This is proved by induction on $\phi$, and gives a proper proof of the previous lemma. 
\end{Remark}

Bearing in mind the lemma above we will call CL-theories of the form $Th_{CL,approx}(M)$, {\em complete CL-theories}.

\begin{Remark} (i) If $T$ is a complete CL-theory, and $M$ is an $\omega$-saturated approximate model of $T$, then $T = Th_{CL}(N)$.
\newline 
(ii)  If $T$ is a complete CL-theory, then there is a closed (so compact) subspace $C_{T}$ of $C$ such that for every approximate model $M = (M^{-},f,C)$ of $T$, $f(M^{-})$ is dense in $C$, and moreover if $M$ is $\omega$-saturated, $f(M^{-}) = C$. 
\end{Remark}

\begin{proof} (i) is obvious from the definitions and Proposition 2.8. 
\newline
For (ii), fix an approximate model $M = (M^{-}, f, C)$ of $T$, and let $C_{0} = cl(f(M^{-}))$. So for every $c\in C_{0}$ and closed neighbourhood $D$ of $\{c\}$, $M\models \exists x (f(x)\in D)$. So the sentence $\exists x (f(x) = c)$ is approximately true in $M$ so is in $T = Th_{CL}(M)$, so approximately true in all models of $T$, and actually true in any $\omega$-saturated model of $T$, by 2.8. 
\end{proof}

So  given the complete CL-theory $T$, and $C_{T}$ as in Remark 2.12(ii), we may assume that in the underlying language $L$, $C = C_{T}$, namely that in a saturated model $f$ is surjective.

There is no harm in assuming that every complete CL-theory $T$  has $\kappa$-saturated models of cardinality $\kappa$ for suitable $\kappa$. And  two such models of $T$ will be isomorphic.  As usual given a complete CL-theory $T$ we feel free to work in such a $\kappa$-saturated model ${\bar M}$ of size $\kappa$ for some very big $\kappa$. 
But we should realize that arbitrary models of $T$ (of cardinality $\leq \kappa$) will in general be only approximately elementarily embeddable in $\bar M$.

\vspace{2mm}
\noindent
We now discuss types and type spaces. 
 For $A\subset \bar M$ and ${\bar b}$ a tuple from $M$, we will compute  $tp({\bar b}/A)$ in ${\bar M}$ and it is just the collection of CL-formulas $\phi({\bar x})$ over $A$ such that
$\bar M \models \phi({\bar b})$.   Note that, as in Remark 2.11 such types are ``maximal consistent". 

\vspace{2mm}
\noindent
{\bf  We now fix a complete CL theory  $T$ and saturated model ${\bar M}$ as above.}
\newline
${\bar x} = (x_{1},..,x_{n})$ is a tuple of variables ranging over the $P$ sort). 

\begin{Definition} $S_{\bar x}(A)$ (or $S_{n}(A)$) is the set of $tp({\bar b}/A)$ for ${\bar b}$ in $\bar M$, equipped with the topology:  $X\subseteq  S_{\bar x}(A)$ is a basic closed set iff there is some CL-formula $\phi({\bar x})$ over $A$ such that  $X = \{tp({\bar b}/A): {\bar M}\models \phi({\bar b})\}$. 
\end{Definition}

Then by the compactness theorem $S_{\bar x}(A)$ is a compact Hausdorff space.  If $A$ is omitted we mean types over $\emptyset$. 

If  $\Phi$ is a family of CL-formulas, then we can talk about
$tp_{\Phi}({\bar b}/A)$ (or sometimes $\Phi$-$tp({\bar b})$): it is the collection of $CL$-formulas with parameters, $\phi({\bar x}, {\bar a})$ such that $\phi({\bar x}, {\bar y})$ is in $\Phi$, ${\bar a}$ is a tuple from $A$ and ${\bar M}\models \phi({\bar b}, {\bar a})$. 
%If $\phi$ contains the quantifier-free formulas and is cloed under $\wedge$, $\vee$, and approxomations, then we again %obtain a compact space $S_{\Phi, {\bar x}}(A)$.

\begin{Definition} We will say that $T$ has ``quantifier elimination" down to formulas in $\Phi$, if whenever ${\bar b}$ and ${\bar c}$ are finite tuples of the same length, and
$tp_{\Phi}({\bar b}) = tp_{\Phi}({\bar c})$ then $tp({\bar b}) = tp({\bar c})$. 
\end{Definition} 

This notion is rather more meaningful when the collection $\Phi$ has some closure properties.  We will call the collection $\Phi$ of CL-formulas closed if $\phi$ contains the quantifier-free formulas, and is closed under $\wedge$, $\vee$, and $\neg$ (i.e. if $\phi$ is a FOL $L^{-}$-formula which is in $\Phi$, then so is $\neg\phi$), as well as approximations.
Fix the tuple ${\bar x}$ of variables. Then $S_{n, \Phi}$ denotes the set of  $tp_{\Phi}({\bar b})$ for ${\bar b}$ an $n$-tuple from ${\bar M}$.  Assuming that $\Phi$ is closed,   $S_{n,\Phi}(T)$ is a compact Hausdorff space when equipped with the topology where a basic closed set is the collection of $\Phi$-types containing a given formula $\phi({\bar x})\in \Phi$.

\begin{Lemma} Assume that $\Phi$ is closed. Then $T$ has quantifier elimination   down to formulas in $\Phi$ if for every CL-formula $\phi({\bar x})$ and approximation to it, $\phi'({\bar x})$ there is a formula $\theta({\bar x})\in \Phi$ such that in any approximate model $M$ of $T$ for any tuple ${\bar a}$ from $M$, we have: if $M\models_{approx} \phi({\bar a})$ then $M\models_{approx}\theta({\bar a})$, and if $M\models_{approx}\theta({\bar a})$ then $M\models_{approx} \phi'({\bar a})$. 
\end{Lemma}
\begin{proof} We will  fix $n$ and consider CL-formulas with $n$ variables and work in the type spaces  $S_{n}(T)$, 
and $S_{n,\Phi}(T)$.  We identify formulas with the sets they define. The ``quantifier elimination" assumption implies that the restriction map from $S_{n}(T)$ to $S_{n,\Phi}(T)$ is a homeomorphism. Let $\phi({\bar x})$ be a CL-formula, so $\phi$ defines a closed set in $S_{n,\Phi}(T)$, namely given by a possibly infinite conjunction $\bigwedge_{i\in I}\psi_{i}({\bar x})$ of $\Phi$-formulas.  If this was a finite conjunction we would be finished, but it need not be, so we require some approximations.  Choose an arbitrary approximation $\phi'$ to $\phi$. As we can always make $\phi'$ smaller, we may assume that $\phi'$ is the closure of an open neighbourhood $U$ of $\phi$  (in the type space).  The complement of $U$ is closed, so compact and covered by the union of the complements of the $\psi_{i}$'s, so by a finite subunion. It follows 
that  for some $n$ we have that  $\phi({\bar x}) \subseteq \bigwedge \psi_{i}({\bar x}) \subseteq \phi'({\bar x})$.
Let $\theta({\bar x}) = \wedge_{i=1,..,n}\psi_{i}({\bar x})$.  So now working with truth of formulas in the saturated model ${\bar M}$ of $T$ we have that $\phi({\bar x})$ then $\theta({\bar x})$ implies $\phi'({\bar x})$.  This suffices, using Proposition 2.8 (iii).

\end{proof}

Finally we discuss stability. We have already discussed types, and thus we obtain the notion of an {\em indiscernible sequence}. 
\begin{Definition}  $T$ is {\em stable} if whenever $({\bar a}_{i},{\bar b}_{i}): i < \omega)$ is an indiscerible sequence
then $tp(({\bar a}_{i}, {\bar b}_{j})) = tp(({\bar a}_{j}, {\bar b}_{i}))$ whenever $i < j$. 
\end{Definition}

One can adapt this definition to give stability of any given CL-formula $\phi({\bar x}, {\bar y})$. 
In any case, the definition makes sense in the classical FOL context where it is correct. 

We now give some general remarks, questions and problems:
\begin{Problem} One possible point of view is that the formalism developed in this section could be a mechanism for adjoining new KP-strong types to a first order theory.  So one should formulate a version of the set-up where any homeomorphism of $C$ with itself lifts to an  automorphism of a saturated model $M$.  It can probaby be done by working, not with all closed subsets of $C^{n}$ but with all closed subsets invariant under homeomorphisms of $C$. 
\end{Problem}

Secondly given a CL-structure $M = (M^{-},f,C)$ what should we mean by the ``induced first order structure on $M^{-}$".

\begin{Definition}  Let $M = (M^{-},f,C)$ be a CL-structure, and $T = Th_{CL, approx}(M)$. Then by the induced $FOL$-structure on $M^{-}$ we mean that obtained by adding predicates corresponding to clopen subsets of the various type spaces $S_{n}(T)$. 
\end{Definition}

The following should be checked: 

\begin{Remark}  Suppose that $M^{-}$ is a saturated FOL structure in language $L^{-}$. Let  $E$ be a bounded $\emptyset$-type-definable equivalence relation on $M^{-}$ (or on some sort in $M^{-}$). Then  adjoining $X/E$ as a new sort adds no first order new structure to $M^{-}$ (modulo working over additional parameters)
\end{Remark}

There has been a considerable amount of work on stable expansions of the structure $(\Z., +)$.  See \cite{Conant} for example. A natural generalization is:
\begin{Problem} What are the stable CL-expansions  $((\Z,+),f,C)$ of $(\Z,+)$?
\end{Problem}

It will be routine to define the notion of $1$-basedness for a stable CL-theory $T$. For example, working in a saturated model $N$, and a stationary type  $tp({\bar a}/M)$ where $M$ is an (approximate) elementary substructure of $N$, the definitions of this type are over the bounded closure of ${\bar a}$
 The content of the main result of this paper is that the CL-theory of an abelian structure equipped with a homomorphism to a compact Hausdorff group is not only stable but also $1$-based.
 It would be nice to have a converse, extending the theorems about $1$-based groups in FOL to the current CL context.
\begin{Problem} Suppose that $G$ is an FOL expansion of a group, and $f:G\to C$ is a map to a compact {\em space} $C$. Suppose that the CL-theory of $M = (G,f,C)$ is stable and $1$-based.  Show that  $G$ is essentially an abelian structure {\em and} up to a suitable sense of piecewise translation $C$ has a group structure compatible with its topology and $f$ is a homomorphism.
\end{Problem}

\vspace{5mm}
\noindent
The version of continuous logic given in the current paper is relatively self contained and should be directly accessible to any model theorist.
But we thought it would be a good idea to give a translation with the treatment in \cite{BY-B-H-U}.  This last part of Section 2 is written ONLY for those conversant with the theory in \cite{BY-B-H-U}, which we will call temporarily \cite{BY-B-H-U}-continuous logic. 

Consider a CL-structure $M = (M^{-},f,C)$ in the sense of this paper, and let us assume for now that the compact space $C$ is precisely the interval $[0,1]$.   So we can view $M$ as a \cite{BY-B-H-U}-continuous logic structure, by thinking of the $L^{-}$ relations as  $[0,1]$-valued relations (which take only values $0$ for true and $1$ for false), and $f$ as a $[0,1]$-valued relation on the universe of $M^{-}$. Moreover the metric on $M^{-}$ is discrete (again with values $0$, $1$). Work in \cite{BY-B-H-U}- continuous logic where we allow all continuous functions from $[0,1]^{n}$ to $[0,1]$ as connectives.  Remember that \cite{BY-B-H-U}- continuous logic formulas are $[0,1]$-valued, whereas our CL-formulas are true or false. 

\begin{Lemma} There is a one-one correspondence, up to equivalence,  between CL-formulas and \cite{BY-B-H-U}- continuous logic formulas  $\phi({\bar x}) \to  \phi^{*}({\bar x})$  such that for any CL-structure  $M = (M^{-},f,[0,1])$,   and tuple ${\bar a}$ from $M$, $M\models_{approx}\phi({\bar a})$ if and only if in $M$ viewed as a \cite{BY-B-H-U}-continuous logic structure, the condition $\phi^{*}({\bar x}) = 0$ holds of ${\bar a}$, that is the value of the formula $\phi^{*}({\bar x})$ at ${\bar a}$ in the structure $M$ is $0$. 
\end{Lemma}
\begin{proof} 
First how to go from CL-formulas to \cite{BY-B-H-U}-continuous logic formulas.
First, view each $FOL$ $L^{-}$-formula as a $\{0,1\}$ valued \cite{BY-B-H-U}-continuous logic formula (as usual).
Now replace $f$ by a $[0,1]$-valued relation $R_{f}$. Given an atomic CL-formula  $(f(x_{1}),.., f(x_{n})\in D$, let $u$ be a continuous function from $[0,1]^{n}$ to $[0,1]$ such that $D = u^{-1}(0)$, and replace any occurrence of  $(f(x_{1}),.., f(x_{n}))\in D$, by the \cite{BY-B-H-U}-continuous logic formula  $u(R_{f}(x_{1}),..,R_{f}(x_{n}))$.
Replace $\wedge$, $\vee$ by $max$, $min$.  And replace $\exists$, $\forall$ by the ``quantifiers" $inf$ and $sup$.

\vspace{2mm}
\noindent
As an example, consider the case of of a CL-formula $\phi({\bar x})$:  $\exists y (\rho({\bar x},y) \wedge f(y)\in D)$. 
Then $\phi^{*}({\bar x})$ is $inf_{y}(max(\rho({\bar x},y), u(R_{f}(y)))$ where $u$ is a continuous function $[0,1]\to [0,1]$ such that $u^{-1}(0) = D$, and as above $\rho$ is now considered as a $\{0,1\}$ valued relation.  
Then  $\phi^{*}({\bar a})= 0$ means that for any small $\epsilon > 0$, there is $b$ such that $M\models \rho({\bar a},b)$ and $u(R_{f}(b))\in [0,\epsilon]$, and via the choice of $u$, this means that for any approximation $D'$ to $D$,  $M\models \exists y(\psi({\bar a},y) \wedge f(y)\in D')$. 

\vspace{2mm}
\noindent
The other direction is similar and left to the reader, except that the continuous function connectives $[0,1]^{n} \to [0,1]$ are replaced by their graphs in $[0,1]^{n+1}$ (and existential quantifiers). 
\end{proof}

\vspace{5mm}
\noindent
When $C$ is an arbitrary compact  space, we can replace $f$ by the collection of all functions $g$ from $M^{-}$ to $[0,1]$, such that for some continuous function $\pi:C\to [0,1]$, $g = \pi\circ f$.

\vspace{5mm}
\noindent
Another option, suggested by Ward Henson, is as follows: if $C$ is a metrizable compact space, then simply view $M$ as a $2$-sorted \cite{BY-B-H-U}- continuous logic structure where the second sort is equipped with its metric as well as all continuous functions to $[0,1]$.  In general, $C$ will be an inverse limit of metrizable compact spaces and add all these as sorts, as well as the connecting maps.

\section{Abelian structures with a homomorphism to a compact group}
Having introduced our formalism in Section 2, we can now be more precise about the results of this paper. 

We will take an {\em abelian structure} $A$  to be an abelian group $(A,+,-,0)$ equipped with a family ${\mathcal S}$ of subgroups of various Cartesian powers of $A$.  We treat $A$ as a structure in a language $L^{-}$  (with equality) which has function symbols $+$, $-$, constant symbol $0$ and predicate (or relation) symbols $S$ for each $S\in {\mathcal S}$. One can identify equality with a distinguished subgroup $S_{=}$ of $A^{2}$, if one wishes.

An atomic $L^{-}$ formula is by definition a formula of the form $P(t_{1},..,t_{n})$ where $t_{1},..,t_{n}$ are terms and $P\in {\mathcal S}$ is $n$-ary.  And by definition a positive primitive $L^{-}$-formula is something of the form 
$\exists {\bar y}\phi({\bar x}, {\bar y})$ where $\phi$ is a finite conjunction of atomic formulas, and we exhibit all the free variables of $\phi$ by the tuple ${\bar x}, {\bar y}$.  

On the other hand given a ring $R$ (with a $1$), the language $L_{R}$ of $R$-modules has again symbols $+,-, 0, 1$ as well as unary function symbols $\lambda_{r}$ for each $r\in R$.  The only relation symbol is equality. One again has the notion of a positive primitive formula Any (left) $R$ module can be considered as an abelian structure by replacing each $\lambda_{r}$ by its graph. Then the positive primitive formulas of $L_{R}$ coincide with those in the corresponding abelian structure language, insofar as their interpretations are concerned.  So in this sense $R$-modules are a special case of abelian structures, and from now on we will only talk about (enriched) abelian structures. 

It is well-known and immediate to check that if $\phi({\bar x}, {\bar y})$ is a $pp$ formula of $L^{-}$, then $\phi({\bar x}, {\bar 0})$ defines a subgroup of $A^{n}$ (a so-called $pp$-subgroup) and that for any tuple ${\bar b}$ from $A$ (of length that of ${\bar y}$), the formula $\phi({\bar x}, {\bar b})$ defines a coset (translate) of the subgroup.

We now add a homomorphism $f$ from $A$ to a compact Hausdorff group which we call $\mathbb T$, and we are interested in $M = (A,f,{\mathbb T})$ as a CL-structure and its theory  $T = Th_{CL}(M)$  as a CL-theory. (So when we write $A$ we have implicitly all its abelian structure.)  So $T$ is a complete CL-theory in a language $L$, which is fixed from now on.  From Remark 2.12 (ii), we may assume that $f(A)$ is dense in ${\mathbb T}$. 

\begin{Remark}  Let $(B,g,{\mathbb T})$ be an approximate model of $T$. Then $(B,+,-,0)$ is an abelian group, the interpretations of the distinguished predicate symbols of $L^{-}$ are subgroups (or the appropriate $B^{n}$), and $g$ is a homomorphism (with dense image).
\end{Remark}
\begin{proof} As the first order $L^{-}$ theory of $B$ is part of $T$, the first two parts are immediate.
For the last part, note that  the expression $\sigma$: $\forall x\forall y(f(x+y) -f(x) - f(y) = 0) $ is a CL-sentence which is in $T$.  If $(B,g, {\mathbb T})$ is an approximate model of $\sigma$, then it is also  a model of $\sigma$, so $g$ is a homomorphism. $g(B)$ being dense follows from 2.12 (ii).
\end{proof}

\begin{Definition}  By a $pp^{*}$-formula we mean informally a $pp$-$L^{-}$-formula, $\exists {\bar x}(\phi({\bar x}, {\bar y}))$ where $\phi$ is a conjunction  of atomic $L^{-}$ formulas, {\em and where in addition} we specify the values under $f$ of none, some, or all, of the variables in ${\bar x}$, ${\bar y}$.   More formally, a $pp^{*}$-formula, is a CL-formula of the following form:  
$$\exists {\bar y}(\phi({\bar x}, {\bar y}) \wedge \bigwedge_{i\in I}(f(x_{i}) = c_{i}) \wedge \bigwedge_{j\in J}(f(y_{j}) = d_{j}))$$
where $\phi({\bar x}, {\bar y})$ is a finite conjunction of atomic $L^{-}$-formulas, $I$ is a subset of the indices of ${\bar x}$, and $J$ a subset of the indices of ${\bar y}$.
\end{Definition}

We now pass to a saturated model  of $T$, which we again call $M = (A,f,{\mathbb T}$ and now $f$ is surjective. 
By a $pp$-formula we mean a $pp$ $L^{-}$ formula. Such a formula defines a subgroup of the relevant $A^{n}$ called a $pp$-subgroup.  By a $pp^{-}$ formula we mean the negation of a $pp$-formula which, note is a CL-formula. Given a $pp^{*}$-formula $\psi({\bar x})$ as in Definition 3.2 where any $c_{i}$, $d_{j}$ which appear are $0$, then 
$\psi({\bar x})$ also defines a subgroup of $A^{n}$, which we may call a $pp^{*}$-subgroup. Of course a $pp$-subgroup is a special case of a $pp^{*}$-subgroup. As in the proof of 2.12 (ii) the image of a $pp^{*}$ subgroup of $A^{n}$ is a {\em closed} subgroup of ${\mathbb T}^{n}$.  

We often identify formulas with the sets the define. We will sometimes let $K$ denote $Ker(f)$, a subgroup of $A$. 
The notions $pp^{*}$-type and $pp$-type are the obvious ones (as with $\Phi$-type in Section 2).

Our main theorem, ``quantifier elimination down to $pp^{*}$-formulas" is:
\begin{Theorem} Suppose $\bar a$ and $\bar b$ are finite tuples from $M$ (namely from the sort $A$) such that  $pp^{*}tp({\bar a}) = pp^{*}tp({\bar b})$ then $tp({\bar a}) = tp({\bar b})$.
\end{Theorem}

Here $tp$ refers to CL-type in the structure $M$. 

Theorem 3.3 will be proved by a standard back-and-forth argument.  We will need the following lemma, the proof of which is a simple version of the proof of the Theorem.

\begin{Lemma}  Let ${\bar a}$, ${\bar b}$ be tuples of the same length in $M$.  Suppose that $pp^{*}tp({\bar a}) \subseteq pp^{*}tp({\bar b})$ and 
$pp^{-}tp({\bar a}) \subseteq pp^{-}tp({\bar b})$. Then $pp^{*}tp({\bar a}) = pp^{*}tp({\bar b})$.
\end{Lemma}
\begin{proof}  Given the hypotheses we have to prove that $pp^{*}tp({\bar b}) \subseteq pp^{*}tp({\bar a})$.
Let $\psi({\bar x})$ be a $pp^{*}$-formula:  $\exists {\bar y}(\phi({\bar x}, {\bar y})\wedge {\bar f}({\bar x}) = {\bar r} \wedge  {\bar f}({\bar y}) = {\bar s})$, where notation has the obvious interpretation (in particular we are specifying the value under $f$ of all variables). 
And suppose that $M\models \psi({\bar b})$. First note that the value of ${\bar a}$ under ${\bar f}$ is part of $pp^{*}tp({\bar a})$, so holds of ${\bar b}$ too, whereby  
\newline
(*)    ${\bar f}({\bar a}) = {\bar r}$. 

\vspace{2mm}
\noindent
 Now  $\exists{\bar y}(\phi({\bar x},{\bar y}))$ is a $pp$-formula true of ${\bar b}$. Our assumptions imply that this $pp$-formula is true of ${\bar a}$.  So the $pp$-formula {\em with 
parameters}  $\phi({\bar a}, {\bar y})$ is consistent, hence defines a coset  (in suitable $A^{m}$) of the subgroup of $A^{m}$ defined by $\phi({\bar 0},{\bar y})$.  By remarks above
${\bar f}(\phi({\bar 0}, {\bar y}))$ is a closed subgroup $H$ of ${\mathbb T}^{m}$, and  ${\bar f}(\phi({\bar a}, {\bar y}))$ defines a coset (translate) $C$ of $H$ in ${\mathbb T}^{m}$. 
But $\phi({\bar b}, {\bar y})$ is nonempty, so also ${\bar f}(\phi({\bar b}, {\bar y}))$ is a coset $D$ say of $H$ in ${\mathbb T}^{m}$. 
Pick any  ${\bar s}'$ in $C$. So $M \models \exists {\bar y}(\phi({\bar a}, {\bar y}) \wedge {\bar f}({\bar y}) = {\bar s}')$.   By our assumptions
$M \models \exists {\bar y}(\phi({\bar b}, {\bar y}) \wedge {\bar f}({\bar y}) = {\bar s}')$ too.  This means that $C\cap D \neq \emptyset$, hence $C = D$, as both are cosets of the same 
subgroup $H$ of ${\mathbb T}^{m}$.  Hence, as  ${\bar s}\in D$, ${\bar s}\in C$ too, hence $M\models \exists {\bar y}(\phi({\bar a}, {\bar y})\wedge {\bar f}({\bar y}) = {\bar s})$. Together 
with (*) this shows that $M\models \psi({\bar a})$, which completes the proof of the lemma. 

\end{proof}

\vspace{5mm}
\noindent
Now we can give:
\newline
{\em Proof of Theorem 3.3.}
\newline
This will be closely related to a standard proof of $pp$ elimination in modules (see \cite{Ziegler}, \cite{Gute}) and recovers that result. 

As the $pp^{*}$-type of a tuple contains the quantifier-free type of that tuple, it will be enough to prove the following ``back-and-forth" statement:

\begin{Lemma}  Suppose that ${\bar a}$ and ${\bar b}$ are finite tuples from $M$ (or rather $A$) of the same length and with the same $pp^{*}$-type. Let $c\in A$. Then there is $d\in A$ such that $({\bar a}, c)$ and $({\bar b},d)$ have the same $pp^{*}$-type. 
\end{Lemma}
\begin{proof} Let  $q({\bar x}, y) = pp^{*}tp({\bar a}, c)$.  and let $q_{0} = pptp({\bar a}, c)$.
If we can find $d$ satisfying $q({\bar b},y) \cup \{\neg\psi({\bar b},y):\psi({\bar x},y)\notin q_{0}\}$, then by Lemma 3.4 we will have that $pp^{*}tp({\bar a},c) = pp^{*}tp({\bar b},d)$. 

\vspace{2mm}
\noindent
So we will prove:
\newline 
{\em CLAIM.}   $q({\bar b},y) \cup \{\neg\psi({\bar b},y):\psi({\bar x},y)\notin q_{0}\}$  is consistent (i.e. realized in $M$). 
\newline
{\em Proof of CLAIM.}
By compactness it is enough to show that any finite subset of the above set of CL-formulas is consistent. As $q$ is closed under finite conjunctions we must show that for any $pp^{*}$-formula  $\chi({\bar x},y)$ in $q$ and  $pp$-formulas $\psi_{1}({\bar x},y)$,..,$\psi_{n}({\bar x}, y)$, not in $q_{0}$, that:

\vspace{2mm}
\noindent
(*) $ \chi({\bar b},y) \wedge \neg\psi_{1}({\bar b},y) \wedge ... \wedge \neg\psi_{n}({\bar b},y)$ has a solution in $M$. 

Identifying formulas with their solutions in $M$ this means we have to show that $\chi({\bar b},y)$ is not covered by the union of the $\psi_{i}({\bar b},y)$  ($i=1,..,n$).

Let us write explicitly  $\chi({\bar x},y)$ as  $\exists {\bar z}(\phi({\bar x},y,{\bar z}) \wedge ({\bar f}({\bar x}) = {\bar r}) \wedge (f(y) = s) \wedge  ({\bar f}({\bar z}) = {\bar t}))$.  (Note we can assume that all values of the variables $x_{i}, y, z_{i}$ under $f$ are specified in $\chi$.)

Then $\phi({\bar b},y)$ defines a coset $X$ of the subgroup $H$ of $K$ defined by 
\newline
$\exists {\bar z}(\phi({\bar 0},y,{\bar z}) \wedge (f(y) = 0) \wedge  ({\bar f}({\bar z}) = {\bar 0}))$. On  the other hand $\psi_{i}({\bar b},y)$ defines
a coset of the subgroup $H_{i}$ of $A$ defined by $\psi_{i}({\bar 0},y)$. 

Now the Neumann lemma (see Lemma 1.4 of \cite{Ziegler}) says that $X$ is covered by the $X_{i}$'s iff $X$ is covered by those $X_{i}$'s such that $H\cap H_{i}$ has finite index in $H$.   So we may assume that for each $i$, $H\cap H_{i}$ has finite index in $H$. 

On the other hand we have the following elementary inclusion-exclusion principle (Fact 3.6 below). 
\newline
Let us denote  $H\cap \bigcap_{i=1,..,n}H_{i}$ by $K_{0}$, a finite index subgroup of $H$.  So for each $\Delta\subseteq 
\{1,..,n\}$, $\bigcap_{i\in \Delta}X\cap X_{i}$ is a finite (disjoint) union of cosets of $K_{0}$ and by $|\bigcap_{i\in 
\Delta}X\cap X_{i}|$ we will temporarily mean the number of these cosets (namely the ``index"  of $K_{0}$ in $
\bigcap_{i\in {\Delta}} X\cap X_{i}$).  We can use the same notation for $H$ and $H_{i}$ in place of the  $X$ and $X_{i}
$.   And let us note for now that  

\vspace{2mm}
\noindent
(**) $|\bigcap_{i\in \Delta}X\cap X_{i}| = |\bigcap_{i\in \Delta}H\cap H_{i}|$ if and only if 
$\bigcap_{i\in \Delta}X\cap X_{i} \neq \emptyset$. 

Then we have:
\begin{Fact}  $X$ is contained in $\bigcup_{i\in \Delta}X_{i}$ iff 
$\sum_{\Delta\subseteq \{1,..,n\}} (-1)^{|\Delta|} |X\cap \bigcap_{i\in \Delta}X_{i}| = 0$
\end{Fact}

Now replacing ${\bar b}$ by ${\bar a}$ we obtain $\chi({\bar a}, y)$ (which we call $Y$) and $\psi_{i}({\bar a}, y)$ which we call $Y_{i}$, cosets of $H$ and $H_{i}$ respectively. Our assumptions apply to give that $Y$ is NOT covered by the union of the $Y_{i}$. On the other hand Fact 3.6 also applies to the $Y$ and $Y_{i}$. Now (*) says exactly that $X$ is not covered by the $X_{i}$. Hence to prove this, it is enough, using (**) to prove:

\vspace{2mm}
\noindent
{\em Subclaim.}  For each $\Delta\subseteq \{1,..,n\}$,  $X\cap\bigcap_{i\in \Delta} \neq \emptyset$ iff $Y\cap \bigcap_{i\in \Delta}Y_{i} \neq \emptyset$. 
\newline
{\em Proof of subclaim.}  The right hand side is expressed by a $pp^{*}$-formula true of ${\bar a}$, so true of ${\bar b}$ by  our hypotheses, yielding the left hand side.

Conversely, suppose the right hand side fails. 
\newline
{\em Case (i).} $M\models \neg (\exists y) (\exists {\bar z})(\phi({\bar a}, y, {\bar z}) \wedge \bigwedge_{i\in \Delta}\psi({\bar a}, y))$. 
\newline
Then as this means that a certain negated $pp$-formula is true of ${\bar a}$, by our hypotheses it is also true of ${\bar b}$, so a fortiori, the left hand side of the Subclaim fails.
\newline
{\em Case (ii).} $M\models (\exists y) (\exists {\bar z})(\phi({\bar a}, y, {\bar z}) \wedge \bigwedge_{i\in \Delta}\psi({\bar a}, y))$. 
\newline
Then the set of $(y,{\bar z})\in A^{1+k}$ such that $\phi({\bar a}, y, {\bar z}) \wedge \bigwedge_{i\in \Delta}\psi({\bar a}, y)$  holds in ${\bar M}$, is a coset (translate) $Z$ of the subgroup $S$ of $A^{1+k}$ defined 
by $\phi({\bar 0},y,{\bar z}) \wedge \bigwedge_{i\in \Delta}\psi_{i}({\bar 0},y)$.  Then ${\bar f}(Z)\subseteq {\mathbb T}
^{1+k}$ is a coset of the (closed) subgroup ${\bar f}(S)$ of ${\mathbb T}^{1+k}$. 
Now our assumption that the right hand side in the subclaim fails means that  $(s,{\bar t})$ (from the formula $\chi$) is not in ${\bar f}(Z)$.  Let $(s',{\bar t}')$ be any point in ${\bar f}(Z)$. Then the $pp^{*}$- formula   $\exists y \exists {\bar z}(\phi({\bar x},y,{\bar z})\wedge \bigwedge_{i\in \Delta}\psi_{i}({\bar x},y) \wedge f(y) = s' \wedge {\bar f}({\bar z}) = {\bar t}')$ is true of ${\bar a}$, so by assumption true of ${\bar b}$. But  the set of $(y,{\bar z})\in A^{1+k}$ such that $\phi({\bar b}, y, {\bar z}) \wedge \bigwedge_{i\in \Delta}\psi({\bar b}, y)$ holds,  is  a coset $Z'$ of $S$, whereby  ${\bar f}(Z')$ is also a coset of the subgroup ${\bar f}(S)$. We have just seen that  $(r',{\bar t}')\in {\bar f}(Z)\cap {\bar f}(Z')$. Hence  ${\bar f}(Z) = {\bar f}(Z')$, and so $(r,{\bar t})\notin {\bar f}(Z')$ which proves that the left hand side of the statement of the subclaim fails. 

\vspace{2mm}
\noindent
This proves the subclaim, the claim, as well as the lemma.
\end{proof}

As remarked above, this back and forth property for equality of $pp^{*}$-types suffices, by a routine argument to prove  Theorem 3.3

\qed (of proof of Theorem 3.3)

Let $\Phi$ be the collection of CL-formulas obtained from the $pp^{*}$ formulas and negated $pp$-formulas by closing under $\wedge$, $\vee$, and approximations.

\begin{Remark}  From Lemma 2.15 we obtain for $T$: for any formula $\phi({\bar x})$ and approximation to it $\phi'({\bar 
x})$ there is a formula $\theta({\bar x})\in \Phi$ such that in a saturated model, $\phi({\bar x})$ implies $\theta({\bar x})$ implies $\theta'({\bar x})$ (and likewise with approximate satisfaction in arbitrary models of $T$). 
\end{Remark}

\begin{Corollary}  $T$ is stable.

\end{Corollary}
\begin{proof}  This is very standard, but we do it anyway because of the somewhat different set-up. Working in the saturated model, let $(({\bar a}_{i},{\bar b}_{i}): i < \omega)$ be an indiscernible sequence of pairs of finite tuples (from $A$).  We have to show that if $i< j$ then $({\bar a}_{i}, {\bar b}_{j})$ and $({\bar a}_{j},{\bar b}_{i})$ have the same type

By Theorem 3.3 we have to show that they satisfy the same $pp^{*}$-formulas. Let $\phi^{*}({\bar x}, {\bar y})$ be a $pp^{*}$-formula  
$\exists {\bar z}(\phi({\bar x}, {\bar y}, {\bar z}) \wedge ({\bar f}({\bar x}) = {\bar r}) \wedge ({\bar f}({\bar y}) = {\bar s}) \wedge ({\bar f}({\bar z}) = {\bar t}))$.  Then 
for any ${\bar b}$, $\phi^{*}({\bar x}, {\bar b})$, if consistent, defines a coset of the group defined by
$\exists{\bar z}(\phi({\bar x}, {\bar 0}, {\bar z}) \wedge {\bar f}({\bar x}) = {\bar 0} \wedge {\bar f}({\bar z}) = {\bar 0})$.   And two such cosets are equal or disjoint.
Let us assume that  $M\models \phi^{*}({\bar a}_{j}, {\bar b}_{i})$ and show $M\models \phi^{*}({\bar a}_{i}, {\bar b}_{j})$
We may assume that there is $k$ with $i < k < j$.
Let $j' > j$.   By indiscernibility $M \models \phi^{*}({\bar a}_{j'}, {\bar b}_{j})$. So $\phi^{*}({\bar x}, {\bar b}_{i})$ and $\phi^{*}({\bar x}, {\bar b}_{j})$ are equivalent (define the same set).   We also have, by indiscernibility, that $M\models \phi^{*}({\bar a}_{k}, {\bar b}_{i})$.  Hence also $M\models \phi^{*}({\bar a}_{k}, {\bar b}_{j})$, and so again by indiscernibility, we have $M\models \phi^{*}({\bar a}_{i}, {\bar b}_{j})$.  Good.

\vspace{2mm}
\noindent
By a symmetric argument we have $M\models \phi^{*}({\bar a}_{i}, {\bar b}_{j})$ implies $M\models \phi^{*}({\bar a}_{j}, {\bar b}_{i})$. 

\end{proof}

\begin{Example} Let $f$ be a homomorphism from $(\Z,+)$ to the circle group $S_{1}$ by mapping $1$ to an ``irrational rotation". So $f(\Z)$ is dense in $S_{1}$. Then  the approxmate CL-theory of $((\Z,+), f, S_{1})$ is stable.
\end{Example}

In \cite{Tran-Walsberg} it is shown that expanding $(\Z,+)$ by the preimage of a small interval around the identity of $S_{1}$ under $f$ is dp-miniimal, hence by \cite{Conant} unstable, {\em as a first order structure}.  There is no contradiction as,  there will be no additional induced first order structure on $(\Z,+)$, in the sense of Definition 2.18.

Finally, it might be natural to ask whether, if $G$ is  a first order expansion of a group, $C$ is a compact group, $f$ is a 
homomorphism from $G$ to $C$ and $(G,f,C)$ is stable as a CL-structure, then this is explained by the existence of a stable first order expansion $G'$ of $G$ such that $C$ is a quotient of $G'/(G')^{00}$ (after passing to a saturated model). However then $C$ has to be profinite (as $(G')^{00} = (G')^{0}$).  Example 3.9 gives a counterexample.

\end{document}